\documentclass[a4paper,francais]{smfart}
\usepackage[T1]{fontenc}
 \usepackage[francais, english]{babel}
 \usepackage{amssymb,url,xspace,smfthm}
\usepackage{smfthm}
 \usepackage{graphicx}
\theoremstyle{plain}

\author[K. DIARRA]{Karamoko DIARRA}
\address{IRMAR, Unit\'e Mixte de Recherche 6625 du CNRS, Universit\'e de Rennes 1, Campus de Beaulieu, 35042 Rennes Cedex, France \\ DER de Math\'ematiques et d'informatique, FAST, Universit\'e de Bamako, BP : E $3206$ Mali.}
\email{diarak2005@yahoo.fr}
 \email{karamoko.diarra@univ-rennes1.fr}
\title[Syst\`eme de Garnier]{Construction et classification de certaines solutions alg\'ebriques des syst\`emes de Garnier}
\usepackage[all]{xy}

\begin{document}
\frontmatter
%\author{Karamoko DIARRA}
%\title{}
%\date{}

%\maketitle
%$\mathit{Karamoko\\\ DIARRA}$, \\{\small ;\\}\\E-mail :   ou karamoko.diarra@univ-rennes1.fr  

%\titlepage{K.Diarra}
%\pagestyle{fancyplain}
%\rhead[\leftmark]{K. DIARRA}

%{\center \section*{R\'esum\'e}}
\begin{abstract}

Dans cet article, nous classifions toutes les solutions alg\'ebriques (compl\`etes) non \'el\'ementaires des syst\`emes de Garnier  construites par la m\'ethode de Kitaev : elles se d\'eduisent des d\'eformations isomonodromiques donn\'ees
en tirant en arri\`ere une \'equation fuchsienne donn\'ee $E$ par une famille de rev\^etements ramifi\'es $\Phi_t$. Nous introduisons tout d'abord les structures orbifoldes associ\'ees et sous-jacentes \`a une \'equation fuchsienne. Ceci nous permet d'avoir une version raffin\'ee de la formule de Riemann Hurwitz qui nous permet rapidement de montrer que $E$ doit \^etre hyperg\'eom\'etrique. Ensuite, on arrive \`a borner le degr\'e de $\phi$ et les exposants, puis enfin \`a lister tous les cas possibles. Ceci g\'en\'eralise un r\'esultat d\^u \`a C. Doran dans le cas de l'\'equation de Painlev\'e VI. Nous construisons explicitement une de ces solutions.
\end{abstract}
%\begin{altabstract}
%hRsum dans l'autre langue (franais ou anglais)i
%\end{altabstract}
\altkeywords{\'Equations diff\'erentielles ordinaires, D\'eformations isomonodromiques, Familles de Hurwitz}
%\translator{hPrnom Nomi}
%\thanks{hSubventionsi}
%\dedicatory{hDdicacei}
\maketitle
%\tableofcontents 
\mainmatter

{\center \section{Introduction}}
 Rappelons qu'un syst\`eme de Garnier de rang $N$ est un syst\`eme d'\'equations diff\'erentielles d'ordre $2$ non lin\'eaire de rang $N$. D'apr\`es un th\'eor\`eme de Garnier \cite{Garnier1}, il est obtenu par la d\'eformation isomonodromique de l'\'equation diff\'erentielle fuchsienne $E$ sur $\Bbb P^1$ avec $2N+3$ p\^oles ($N+3$ points singuliers non apparents et $N$ singularit\'es apparentes), qui s'\'ecrit sous la forme normale  $\frac{d^2u}{dx^2}+f(x)\frac{du}{dx}+g(x)u=0$ avec$$\left\{ \begin{array}{ll}
f(x)=&\frac{1-\theta_0}{x}+\frac{1-\theta_1}{x-1} +\sum_{i=1}^N\frac{1-\theta_{t_i}}{x-t_i} +\sum_{i=1}^N\frac{-1}{x-q_i} \\
g(x)=&\frac{a}{x}+\frac{b}{x-1} -\sum_{i=1}^N\frac{H_i}{x-t_i} +\sum_{i=1}^N\frac{P_i}{x-q_i}
\end{array}\right..$$
Les solutions des syst\`emes de Garnier sont en g\'en\'eral tr\`es transcendantes, mais ils peuvent avoir certaines solutions alg\'ebriques. Dans ce papier, on veut classer et construire explicitement certaines de ces solutions alg\'ebriques de Garnier. Si on commence par le cas le plus simple lorsque le rang $N=1$, le syst\`eme de Garnier se r\'eduit \`a l'\'equation diff\'erentielle de Painlev\'e VI. 
 On ne peut pas exprimer la solution g\'en\'erale de Painlev\'e VI \`a l'aide de solutions d'\'equations lin\'eaires (ou m\^eme non lin\'eaires du premier ordre) et d'op\'erations alg\'ebriques : Watanabe (dans \cite{Watanabe}) montre que toute solution est ainsi transcendante except\'e les solutions de type Riccati, et les solutions alg\'ebriques. Les premi\`eres apparaissent lorsque l'on consid\`ere les d\'eformations isomonodromiques d'\'equations fuchsiennes r\'eductibles (i.e. \`a monodromie triangulaire) : les coefficients de l'\'equation sont peu transcendants et satisfont une \'equation de Riccati. Watanabe \cite{Watanabe} les a classifi\'e, elles apparaissent pour des choix particuliers de valeurs propres. Les solutions alg\'ebriques apparaissent soit en familles, soit de mani\`ere sporadique et, si on en connaissait depuis Picard, la classification compl\`ete vient seulement d'\^etre achev\'ee par Lisovyy et Tykhyy \cite{Lisovyy}.
Pour \^etre plus exact, elle r\'esulte de constructions de Picard, Hitchin (\cite{Hitchin}, \cite{Hitchin1}), Dubrovin, Mazzocco (\cite{Dubr}, \cite{Kapov}), Kitaev (\cite{Belyi}, \cite{Ki}, \cite{K2}) et surtout Boalch (\cite{P2}, \cite{P1}) qui a largement contribu\'e \`a compl\'eter la liste qui s'est ensuite av\'er\'ee d\'efinitive : O. Lisovyy et Y. Tykhyy ont montr\'e qu'il n'y en avait pas d'autres. On trouve (modulo sym\'etries de Painlev\'e VI)
\begin{itemize}
\item $3$ familles \`a un ou deux param\`etres complexes, 
\item une famille \`a param\`etre discret (points de torsion sur une courbe elliptique),
\item $52$ solutions sporadiques.
\end{itemize}
Bien que la motivation initiale de Painlev\'e \'etait de trouver de nouvelles transcendantes, les solutions alg\'ebriques ont elles aussi un int\'er\^et. Elles permettent chez Doran \cite{Doran} de construire des fibrations elliptiques avec connexion de Gauss-Manin explicite, chez Boalch \cite{P} de construire des solutions explicites au probl\`eme de Riemann-Hilbert, de construire des \'equations uniformisantes explicites pour certaines orbifoldes arithm\'etiques. D'apr\`es un travail de Dubrovin et Mazzocco \cite{Kapov}, la transcendance de la solution g\'en\'erale des syst\`emes de Garnier augmente avec $N$ en ce sens qu'on ne peut int\'egrer le syst\`eme de rang $N+1$ avec les solutions du syst\`eme de rang $N$.
Pour autant, il y aura encore des solutions sp\'eciales, moins transcendantes qu'attendues, en consid\'erant par exemple les d\'eformations d'\'equations fuchsiennes r\'eductibles, ou encore \`a monodromie finie. Les premi\`eres, \'etudi\'ees par exemple dans le livre de Iwasaki, Kimura, Shimomura et Yoshida \cite{Isawaki}, se ram\`enent \`a des solutions d'\'equations  diff\'erentielles lin\'eaires. Les secondes vont produire des solutions alg\'ebriques du syst\`eme de Garnier ; c'est d'ailleurs comme cela que Boalch retrouve $49$ des $52$ solutions sporadiques de Painlev\'e VI. On pourrait tenter de classifier ce type de solutions alg\'ebriques, par exemple pour Garnier avec $N=2$, il y en aura beaucoup, mais ce sont plut\^ot les solutions alg\'ebriques correspondant \`a des d\'eformations d'\'equations fuchsiennes non \'el\'ementaires (i.e. \`a monodromie Zariski dense) qui ont un int\'er\^et. On les appellera ``solutions alg\'ebriques non \'el\'ementaires''.
Dans le cas Painlev\'e VI ($N=1$), l'\'equation d'isomonodromie a de nombreuses sym\'etries, et partant d'une des $49$ solutions alg\'ebrique \'el\'ementaires construites par Boalch, on peut d\'eduire par sym\'etrie une autre solution alg\'ebrique, qui sera bien souvent non \'el\'ementaire.  
Le but principal de ce papier est de construire des solutions alg\'ebriques non \'el\'ementaires des syst\`emes de Garnier pour $N>1$ ; aucune n'\'etait connue avant.
La m\'ethode que nous allons utiliser est d\^ue \`a Kitaev \cite{Belyi} : elle permet de retrouver toutes les solutions alg\'ebriques de Painlev\'e VI modulo sym\'etries. L'id\'ee de d\'epart est extr\`emement simple ; elle ne n\'ecessite m\^eme pas de comprendre les \'equations d'isomonodromies.
On fixe une \'equation fuchsienne disons $E$ sur $\mathbb P^1$, puis on la tire en arri\`ere par une famille de rev\^etements ramifi\'es $\phi_t:\mathbb P^1\to \mathbb P^1$. On obtient alors une famille d'\'equations fuchsiennes $E_t:=\phi_t^*E$ qui, localement dans le param\`etre $t$, l\`a o\`u la d\'eformation est topologiquement triviale, est \'evidemment isomonodromique : la monodromie est essentiellement celle de l'\'equation $E$ en bas.
Maintenant, si l'on prescrit le type topologique du rev\^etement, la famille $\phi_t$ forme une famille de Hurwitz : elle est alg\'ebrique et la d\'eformation $E_t$ sera elle-m\^eme \`a coefficients alg\'ebriques. La difficult\'e principale pour mettre en pratique cette construction est que la dimension de d\'eformation est en g\'en\'eral trop petite face au nombre de p\^oles qu'aura $E_t$ : on ne construira qu'une solution alg\'ebrique partielle de l'\'equation d'isomonodromie, la solution compl\`ete restant en g\'en\'eral transcendante. Si $E$ a $n$ p\^oles sur $\mathbb P^1$ et $\phi_t$ de degr\'e $d$, alors $E_t$ aura en g\'en\'eral $nd$ p\^oles ; pour qu'elle soit \'el\'ementaire, on doit avoir $n\ge 3$ et l'espace de d\'eformation de $E_t$ sera de dimension $N(E_t)=nd-3\ge 3d-3$. Pourtant, le nombre de param\`etres libres dans la construction sera born\'e par le nombre de valeurs critiques de $\phi_t$, c'est-\`a-dire par $2d-2$ d'apr\`es Riemann-Hurwitz : $2d-2<<3d-3=N(E_t)$ (on veut \'evidemment $d>1$).
Pour obtenir une d\'eformation alg\'ebrique compl\`ete, il faudra que $\phi_t$ ramifie suffisamment au dessus
des p\^oles pour ne pas en avoir trop en haut ; pour autant, elle doit ramifier suffisamment en dehors pour garder
des param\`etres libres dans la construction.

Consid\'erons un exemple simple pour fixer les id\'ees. Prenons pour $E$ une \'equation hyperg\'eom\'etrique 
(\'equation fuchsienne sur $\mathbb P^1$ \`a $n=3$ p\^oles simples en $z=0,1,\infty$) et consid\'erons
la famille de rev\^etements doubles $\phi_t:\mathbb P^1\to\mathbb P^1$ qui ramifie au dessus d'un des p\^oles,
disons $z=\infty$, et d'un param\`etre libre $z=t$. Alors l'\'equation fuchsienne $E_t=\phi_t^*E$ aura 
(pour $t\not=0,1,\infty$) $5$ p\^oles, \`a savoir $2$ au dessus de $z=0$, $2$ au dessus de $z=1$, 
et $1$ au dessus de $z=\infty$ ; on n'aura qu'une solution partielle (de codimension $1$) au syst\`eme de Garnier 
correspondant (qui est ici de rang $N=2$). Par contre, si l'exposant de l'\'equation $E$ en $z=\infty$
est $\frac{1}{2}$ (i.e. la monodromie projective locale est d'ordre $2$), alors la singularit\'e de $E_t$ au dessus
sera apparente (sans monodromie locale) et on pourra la chasser par une transformation de jauge m\'eromorphe.
On se retrouve alors avec une d\'eformation \`a $1$ param\`etre d'\'equation fuchsienne $E'_t$ \`a $4$ p\^oles.
Il est facile de v\'erifier que le birapport des $4$ p\^oles varie avec $t$ : c'est une vraie d\'eformation.
C'est ainsi qu'on retrouve la famille de solutions alg\'ebriques de Painlev\'e VI \`a $2$ param\`etres 
(\`a savoir les exposants des p\^oles $z=0$ et $z=1$).

Il r\'esulte d'un th\'eor\`eme de Klein \cite{Klein} que toute \'equation fuchsienne \`a monodromie finie est, modulo transformation
de jauge, le tir\'e-en-arri\`ere $\phi^*E$ par un rev\^etement ramifi\'e d'une \'equation hyperg\'eom\'etrique (plus pr\'ecis\'ement dans la liste de Schwarz \cite{Hodgkinson}, \cite{Uniformisation}). En particulier, toutes les solutions alg\'ebriques \'el\'ementaires construites par Boalch (et plus g\'en\'eralement leurs analogues pour les syst\`emes de Garnier) s'obtiennent aussi avec la m\'ethode de Kitaev. Mais comme l'illustre l'exemple pr\'ec\'edent, elle permet aussi de construire des solutions alg\'ebriques non \'el\'ementaires. Doran a classifi\'e toutes celles que l'on pouvait construire de cette mani\`ere dans le cas de Painlev\'e VI : il retrouve les $3$ familles (\`a un ou deux param\`etres complexes) et $4$ solutions sporadiques. Ces derni\`eres s'obtiennent en tirant en arri\`ere les hyperg\'eom\'etriques
d'exposants $(\frac{1}{2},\frac{1}{3},\frac{1}{7})$ et $(\frac{1}{2},\frac{1}{3},\frac{1}{8})$ (groupes de triangles arithm\'etiques) par des rev\^etements de degr\'e $10$, $12$ et $18$. On montre que seuls les syst\`emes de Garnier de rang $N\le3$ ont des solutions compl\`etes non \'el\'ementaires construites par la m\'ethode de Kitaev. Elles sont list\'ees dans les sections $3$ et $5$. L'une d'elles est construite explicitement dans la derni\`ere section. C'est une partie de ma th\`ese que j'ai fait sous la direction de F. Loray \`a l'universit\'e de Rennes $1$.
\section{Structure orbifolde sous-jacente \`a une \'equation fuchsienne}

\subsection{Structure orbifolde formelle}
Soit $X$ une surface de Riemann compacte. Une structure obifolde formelle sur $X$ est la donn\'ee d'une famille finie $\{(t_i, p_i)\}_{i=1,\cdots, n}$ telle que :
\begin{itemize}
\item $t_1,\ldots,t_n\in X$ sont des points deux \`a deux distincts de $X$,
\item $p_1,\ldots,p_n\in \{2,3,4,5,\ldots,+\infty\}$ sont des poids.
\end{itemize}
On peut encore la d\'efinir par une application 
$$p:X\to \mathbb N^*\cup\{\infty\}$$
qui vaut $1$ sauf pour un nombre fini de points, ici $\{t_1,\ldots,t_n\}$, son support.

On aura aussi \`a consid\'erer des structures orbifoldes g\'en\'eralis\'ees o\`u cette fois
$$p:X\to \mathbb Q_+^*\cup\{\infty\}$$
(valant $1$ sauf pour un nombre fini de points). Pour \'eviter toute confusion, on parlera de structure orbifolde enti\`ere concernant la premi\`ere notion.

On d\'efinit la caract\'eristique d'Euler de la courbe orbifolde $\chi(X,p)$ par
$$\chi(X,p):=2-2g+\sum_{t\in X}(\frac{1}{p(t)}-1)=2-2g-n+\frac{1}{p_1}+\cdots+\frac{1}{p_n}$$
o\`u $g=g(X)$ est le genre de la courbe $X$.
Lorsque $p$ est \`a valeurs dans $\{1,\infty\}$, on retrouve la caract\'eristique d'Euler
classique de la courbe \'epoint\'ee $X\setminus\{p=\infty\}$.

Si $\phi:Y\to X$ est un rev\^etement ramifi\'e entre deux surfaces de Riemann compactes
et si $p:X\to \mathbb Q_+^*\cup\{\infty\}$ est une structure orbifolde sur $X$,
alors on d\'efinit le pull-back $\phi^*p$ par 
$$\phi^*p(t):=\frac{p(\phi(t))}{\mathrm{Ind}_\phi(t)}\ \ \ \ \text{pour tout}\ t\in Y$$
o\`u $\mathrm{Ind}_\phi(t)$ est l'indice de $\phi$ en $t$, i.e. $\mathrm{Ind}_\phi(t)=k$
si $\phi(z)=z^k$ pour des coordonn\'ees locales ad\'equates.
Par exemple, si $p$ est la structure triviale $p\equiv 1$ sur $X$, alors $\phi^*p=\mathrm{Ind}_\phi$,
la fonction indicielle.

\begin{prop}[Riemann-Hurwitz version orbifolde]
Si $\phi:Y\to X$ est un rev\^etement ramifi\'e de degr\'e $d$,
alors $\chi(Y,\phi^*p)=d\cdot\chi(X,p)$.
\end{prop}

Si $p$ et $p'$ sont deux structures orbifoldes sur $X$, on dira que $p$ est inf\'erieure
\`a $p'$  et on notera $p\le p'$ lorsque $p(t)\le p'(t)$ pour tout $t\in X$ 
(ici, $\infty=+\infty$ est un infiniment grand). On v\'erifie imm\'ediatement que
$$p\le p'\ \ \ \Rightarrow \ \ \ \chi(X,p)\ge \chi(X,p').$$
(Attention au renversement d'in\'egalit\'es !)

On d\'efinit la structure orbifolde enti\`ere 
sous-jacente $\underline{p}$ \`a une structure orbifolde g\'en\'eralis\'ee
$$p:X\to \mathbb Q_+^*\cup\{\infty\}$$
par 
\begin{itemize}
\item si $p(t)=\frac{n}{q}\not=\infty$ alors $\underline p(t)=n$ o\`u $(n,q)=1$ ;
\item si $p(t)=\infty$ alors $\underline p(t)=\infty$.
\end{itemize}
Autrement dit, $\underline p(t)$ est le plus petit multiple entier de $p(t)$ ; en particulier,
on a 
$$p(t)\le \underline p(t).$$

Avant de voir le lien avec les \'equations fuchsiennes, rappelons le contexte g\'eom\'etrique dans lequel les structures orbifoldes apparaissent naturellement.

\subsection{Structure orbifolde m\'etrique et uniformisation}

Consid\'erons un groupe fuchsien $\Gamma\subset\mathrm{PSL}(2,\mathbb R)$ de type fini agissant proprement discontinument (par isom\'etries) sur le demi-plan de Poincar\'e $\mathbb H$. Supposons en outre $\Gamma$ de co-volume fini, de sorte que $\mathbb H/\Gamma$ est une surface
de Riemann compacte \`a laquelle on a enlev\'e un nombre fini de points.
Notons $\phi:\mathbb H\to\mathbb H/\Gamma\subset X$.
La fonction indicielle $\mathrm{Ind}_\phi$ est constante sur les fibres de $\phi$
et on d\'efinit une structure orbifolde (enti\`ere) sur la courbe $X$ en posant
\begin{itemize}
\item $p(\phi(w)):=\mathrm{Ind}_\phi(w)$ pour tout $w\in\mathbb H$,
\item $p(t)=\infty$ si $t\in X\setminus(\mathbb H/\Gamma)$.
\end{itemize}
Le support de $p$ est pr\'ecis\'ement la r\'eunion des valeurs critiques de $\phi$ et des pointes.
Par construction, $\phi^*p$ (qui se d\'efinit localement comme dans la section pr\'ec\'edente) est la structure
orbifolde triviale sur $\mathbb H$. 

La m\'etrique de Poincar\'e $\mu$ sur $\mathbb H$ descend sur la courbe $X$ en une m\'etrique singuli\`ere 
(\`a courbure constante $-1$ l\`a o\`u elle est lisse).
Le support de la structure orbifolde $p$ est pr\'ecis\'ement la r\'eunion des points singuliers de la m\'etrique :
l'angle de la surface $X$ autour d'un de ses points $t$ est donn\'e par $\frac{2\pi}{p(t)}$.
Rappelons enfin la formule de Gauss-Bonnet dans ce cadre
$$\mathrm{Aire}(X,\mu)=-2\pi\chi(X,p).$$

On d\'efinit de la m\^eme mani\`ere la structure orbifolde d'un quotient de la sph\`ere de Riemann $\mathbb P^1$
par un groupe fini d'isom\'etries pour sa m\'etrique \`a courbure constante $+1$
et on a (attention au signe)
$$\mathrm{Aire}(X,\mu)=2\pi\chi(X,p).$$
On peut enfin consid\'erer les quotients du plan $\mathbb C$ par un groupe d'isom\'etries euclidiennes,
mais la caract\'eristique d'Euler $\chi(X,p)=0$ ne caract\'erise plus l'aire.

\begin{theo}[Klein-Poincar\'e]
Une surface de Riemann compacte munie d'une structure orbifolde enti\`ere $(X,p)$ est uniformisable,
c'est \`a dire correspond \`a un des quotients d\'ecrits au dessus, si et seulement si on n'est pas dans
l'une des situations suivantes :
\begin{itemize}
\item $\mathrm{genre}(X)=0$ et  $\mathrm{support}(p)=\{t\}$ avec $p(t)<\infty$ ;
\item $\mathrm{genre}(X)=0$ et  $\mathrm{support}(p)=\{t_1,t_2\}$ avec $p(t_1)\not=p(t_2)$.
\end{itemize}
De plus, une orbifolde uniformisable est \`a courbure $>0$ (resp. $=0$ ou $<0$) si et seulement si 
$\chi(X,p)>0$ (resp. $=0$ ou $<0$).
\end{theo} 

On en d\'eduit la liste des orbifoldes uniformisables \`a courbure $>0$ :
\begin{itemize}
\item $\mathrm{genre}(X)=0$ et  $\mathrm{support}(p)=\emptyset$,
\item $\mathrm{genre}(X)=0$ et  $\mathrm{support}(p)=\{t_1,t_2\}$ avec $p(t_1)=p(t_2)<\infty$,
\item $\mathrm{genre}(X)=0$ et  $\mathrm{support}(p)=\{t_1,t_2,t_3\}$ o\`u $p$ prend les valeurs :
$$(2,2,k),\ k<\infty,\ \ \  (2,3,3),\ \ \  (2,3,4),\ \ \  \text{ou}\ \ \  (2,3,5)\ ;$$
\end{itemize}
ainsi que celles de courbure nulle :
\begin{itemize}
\item $\mathrm{genre}(X)=0$ et  $\mathrm{support}(p)=\{t\}$ avec $p(t)=\infty$,
\item $\mathrm{genre}(X)=0$ et  $\mathrm{support}(p)=\{t_1,t_2\}$ avec $p(t_1)=p(t_2)=\infty$,
\item $\mathrm{genre}(X)=0$ et  $\mathrm{support}(p)=\{t_1,t_2,t_3\}$ o\`u $p$ prend les valeurs
$$(2,2,\infty),\ \ \  (2,3,6),\ \ \  (2,4,4),\ \ \  \text{ou}\ \ \  (3,3,3),$$
\item $\mathrm{genre}(X)=0$ et  $\mathrm{support}(p)=\{t_1,t_2,t_3,t_4\}$ o\`u $p$ prend les valeurs
$(2,2,2,2)$,
\item $\mathrm{genre}(X)=1$ et  $\mathrm{support}(p)=\emptyset$.
\end{itemize}

Rappelons que les aires $-2\pi\chi$ {d'orbifoldes (uniformisables) hyperboliques ne sont pas arbitrairement petites,
mais born\'ees inf\'erieurement par l'aire $\frac{\pi}{21}$ de l'orbifolde hyperg\'eom\'etrique $(2,3,7)$. Nous listons 
ci-dessous les orbifoldes d'aire $\le\frac{\pi}{3}$ (c'est \`a dire pour lesquelles $-\chi\ge \frac{1}{6}$, on oublie syst\'ematiquement le facteur $2\pi$) ; elles sont toutes hyperg\'eom\'etriques sauf une, de genre $g=0$ avec $n=4$ points orbifoldes.

\begin{table}[htdp]
\begin{center}
\begin{tabular}{|c|c|c|c|c|c|c|c|c|}
\hline 
$(2,3,p)$ &$(2,3,7)$& $(2,3,8)$& $(2,3,9)$& $(2,3,10)$& $(2,3,11)$& $(2,3,12)$&  $\cdots$& $(2,3,\infty)$\\
\hline
$-\chi=\frac{p-6}{6p}$ & $\frac{1}{42}$  & $\frac{1}{24}$ & $\frac{1}{18}$ & $\frac{1}{15}$ & $\frac{5}{66}$ & $\frac{1}{12}$& $\cdots$&  $\frac{1}{6}$\\
\hline

\hline 
$(2,4,p)$ &$(2,4,5)$& $(2,4,6)$& $(2,4,7)$& $(2,4,8)$&   $\cdots$&$\cdots$&$\cdots$& $(2,4,\infty)$\\%&&$(2,5,5)$\\
\hline
$-\chi=\frac{p-4}{4p}$ & $\frac{1}{20}$  & $\frac{1}{12}$ & $\frac{3}{28}$ & $\frac{1}{8}$ & $\cdots$& $\cdots$&$\cdots$& $\frac{1}{4}$\\%&&$\frac{1}{10}$\\
\hline

\hline 
$(2,5,p)$ &$(2,5,5)$& $(2,5,6)$& $(2,5,7)$&   $\cdots$& $\cdots$&$(2,5,\infty)$&&$(2,6,6)$\\
\hline
$-\chi$ & $\frac{1}{10}$  & $\frac{2}{15}$ & $\frac{11}{70}$ & $\cdots$&$\cdots$&  $\frac{3}{10}$&&$\frac{1}{6}$\\
\hline

\hline 
$(3,3,p)$ &$(3,3,4)$& $(3,3,5)$& $(3,3,6)$&    $\cdots$& $(3,3,\infty)$&&$(3,4,4)$ &$(2,2,2,3)$\\
\hline
$-\chi=\frac{p-3}{3p}$ & $\frac{1}{12}$  & $\frac{2}{15}$ & $\frac{1}{6}$& $\cdots$&  $\frac{1}{3}$&&$\frac{1}{6}$ & $\frac{1}{6}$\\
\hline
\end{tabular}
\end{center}
\end{table}
\vspace{1cm}
Enfin, -$\chi$ est minor\'e en fonction du genre $g$ et du nombre $n$ de points orbifoldes :
\begin{table}[htdp]
\begin{center}
\begin{tabular}{|c|c|c|c|c|c|c|c|c|}
\hline 
$(g,n)$ &$(0,3)$&$(0,4)$&$(0,5)$&$(0,6)$&$(1,1)$&$(1,2)$&$(2,0)$\\
\hline
$-\chi\ge$ & $\frac{1}{42}$ & $\frac{1}{6}$ & $\frac{1}{2}$ &$1$ & $\frac{1}{2}$ &$1$& $2$ \\
\hline
\end{tabular}
\end{center}
\end{table}

\subsection{Structure orbifolde et \'equations fuchsiennes}

Soit $E$ une \'equation fuchsienne sur $X$.
On d\'efinit la structure orbifolde de  $E$ (ou plut\^ot de la structure projective induite par $E$ sur $X$)  de la mani\`ere suivante. 
\begin{itemize}
\item $p(t)=1$ si $t\in X$ est un point r\'egulier (i.e. non singulier) de l'\'equation $E$ ;
\item $p(t)=\frac{1}{\vert\theta\vert}\in\mathbb Q^+$ si $t\in X$ est un point singulier de l'\'equation $E$ d'exposant $\theta\in \mathbb Q$ non logarithmique (c'est-\`a-dire \`a monodromie p\'eriodique) ;
\item $p(t)=\infty$ sinon ($t\in X$ est un point singulier de l'\'equation $E$ d'exposant $\theta\not\in\mathbb Q$, ou encore un point singulier logarithmique d'exposant $\theta\in\mathbb Z$).
\end{itemize}
On notera $p(E)$ cette structure orbifolde. On a alors

\begin{prop}
Si $E$ est une \'equation fuchsienne sur $X$ et $\phi:Y\to X$ un rev\^etement ramifi\'e de degr\'e $d$, alors
$$p(\phi^*E)=\phi^*p(E).$$
\end{prop}

La structure orbifolde sous-jacente $\underline{p}(E)$ ne d\'epend que de la monodromie de $E$ :
\begin{itemize}
\item si $t\in X$ est un point r\'egulier (i.e. non singulier) de l'\'equation $E$, alors $\underline{p}(t)=1$ ;
\item si $t\in X$ est un point singulier de l'\'equation $E$, alors $\underline{p}(t)\in \{2,3,4,5,\ldots,+\infty\}$ est l'ordre  de la monodromie locale autour de $t$.
\end{itemize}
Nous observons que la structure orbifolde sous-jacente est invariante par transformations birationnelles sur l'\'equation. Par exemple, $\underline{p}(t)=1$ si $t$ est une singularit\'e apparente.

\begin{prop}
Si $E$ une \'equation fuchsienne sur $X$ et $\phi:Y\to X$ un rev\^etement ramifi\'e de degr\'e $d$, alors
$$\underline p(\phi^*E)\ge\phi^*\underline p(E).$$
En particulier, $\chi(Y,\underline p(\phi^*E))\le d\cdot\chi(X,\underline p(E))$.
\end{prop}

Si une orbifolde enti\`ere $(X,p)$ est uniformisable, i.e. d\'efinie par un rev\^etement ramifi\'e $\phi:U\to X$, $U=\mathbb P^1$, $\mathbb C$ ou $\mathbb H$, alors la d\'eriv\'ee schwarzienne de $\phi^{-1}$ (qui ne d\'epend pas de la d\'etermination choisie) d\'efinit une
\'equation fuchsienne 
$$E\ :\ u''+\frac{S(\phi^{-1})}{2}u=0$$ 
sur $X$ dont la structure orbifolde est pr\'ecis\'ement $p$. Par exemple, on a 
\begin{itemize}
\item si $X=\mathbb P^1$ et $\mathrm{support}(p)=\{0,1,\infty\}$, alors $E$ est l'\'equation hyperg\'eom\'etrique ;
\item si $X=\mathbb P^1$ et $\#\mathrm{support}(p)=4$, alors $E$ est l'\'equation de Heun ;
\item si $X$ est une courbe elliptique et $\#\mathrm{support}(p)=1$, alors $E$ est l'\'equation de Lam\'e.
\end{itemize}
Le r\'esultat principal de cette section est la
\begin{prop}
Soit $E$ une \'equation fuchsienne sur $X$ dont la monodromie est non \'el\'ementaire, i.e. \`a image Zariski dense dans $\mathrm{PGL}(2,\mathbb C)$. Alors sa structure orbifolde enti\`ere sous-jacente $(X,\underline p)$ est une orbifolde uniformisable hyperbolique.
\end{prop}

\begin{proof}Il suffit de montrer que si la structure orbifolde sous-jacente (qui est enti\`ere) n'est pas uniformisable ou n'est pas hyperbolique, alors l'\'equation fuchsienne $E$ est \'el\'ementaire. Rappelons que par transformation de jauge birationnelle, on peut ramener tous les exposants rationnels de $E$ dans l'intervalle $[0,\frac{1}{2}]$ sous l'action du groupe $\langle -\theta,\theta+1\rangle$,
sauf peut-\^etre un que l'on peut ramener dans l'intervalle $[\frac{1}{2},1]$ (sous le m\^eme groupe). Compte-tenu des listes au dessus, on doit donc en outre consid\'erer les cas hyperg\'eom\'etriques $(g,n)=(0,3)$ avec exposants $(\frac{1}{2},\frac{1}{2},\frac{l}{k})$, $(\frac{1}{2},\frac{1}{3},\frac{2}{5})$ et $(\frac{1}{3},\frac{1}{3},\frac{2}{3})$. Mais tous ces triplets sont dans la liste de Schwarz des hyperg\'eom\'etriques \`a monodromie finie, donc \'el\'ementaires eux aussi.
\end{proof}
\subsection{Cons\'equences}
On cherche \`a classifier les couples $(E,\phi)$ o\`u $E$ est une \'equation fuchsienne sur une surface de Riemann $X$ et $\phi:\tilde X\to X$ un rev\^etement ramifi\'e (fini) tels que
\begin{itemize}
\item la monodromie de $E$ est non \'el\'ementaire,
\item  $\phi$ ramifie au dessus de $N\ge 3\tilde g-3+\tilde n$ points distincts entre eux et distincts des singularit\'es non apparentes de $E$
\end{itemize}
o\`u $\tilde g$ est le genre de $\tilde X$ et $\tilde n$ le nombre de singularit\'es non apparentes de $\tilde E$.

La derni\`ere condition nous assure que $\phi$ peut se d\'eformer avec $N$ param\`etres ind\'ependants, c'est \`a dire au moins la dimension de d\'eformation isomonodromique de $\tilde E$ (dimension de l'espace de Teichm\"uller). 
On d\'eduit de la section pr\'ec\'edente les restrictions suivantes
\begin{prop}
Sous les hypoth\`eses ci-dessus, on a $g(X)=g(\tilde X)=0$ et on est dans l'un des cas suivants
\begin{itemize}
\item $E$ a $3$ singularit\'es non apparentes et $\deg(\phi)\le 42$,
\item $E$ a $4$ singularit\'es non apparentes et $\deg(\phi)\le 6$,
\item $E$ a $5$ singularit\'es non apparentes et $\deg(\phi)\le 2$.
\end{itemize}
\end{prop}

\begin{proof}
On a (en supposant au pire des cas $\tilde p_1=\cdots=\tilde p_{\tilde n}=\infty$ pour la structure orbifolde de $\tilde E$)
$$-\chi(\phi^*E)\ge 2\tilde g-2+\tilde n-N\ge 1-\tilde g,$$
et par ailleurs 
$$\chi(\phi^*E)=\deg(\phi)\cdot\chi(E).$$
Puisque l'on veut une d\'eformation non \'el\'ementaire, on veut $\chi(E)<0$ ;
donc $\tilde g=0$ (et donc $g=0$ aussi) et 
$$-\deg(\phi)\cdot\chi(E)\le 1.$$
\'Evidemment, $\deg(\phi)\ge2$ (sinon on n'a pas
de d\'eformation) et il vient 
$$-\chi(E)\le \frac{1}{2}.$$
Maintenant, \'etant donn\'e un couple $(E,\phi)$ satisfaisant aux $3$ conditions pr\'ec\'edentes, on consid\`ere la structure orbifolde sous-jacente, puis l'\'equation fuchsienne uniformisante correspondante $E'$. Alors $(E',\phi)$ satisfait les m\^eme conditions que $(E,\phi)$.
Par ailleurs, si $E'$ satisfait les conclusions de la proposition, il en va de m\^eme de $(E,\phi)$. Donc on suppose dans la suite $E$ uniformisante. On conclut avec les estimations de la section pr\'ec\'edente.
\end{proof}

Les estimations obtenues sont encore tr\`es grossi\`eres. Examinons en d\'etail les cas o\`u $E$ a $4$ ou $5$ singularit\'es non apparentes.
\begin{prop}
Sous les hypoth\`eses pr\'ec\'edentes, $E$ a au plus $3$ singularit\'es.
\end{prop}
\begin{proof}
Comme dans la preuve de la pr\'ec\'edente proposition, on suppose sans perte de g\'en\'eralit\'e $E$ uniformisante et hyperbolique.
Si $n=n(E)=4$, soit $p$ le maximum de sa structure orbifolde :
$$-\chi(E)\ge \left(-2+\sum_1^3(1-\frac{1}{2})+(1-\frac{1}{p})\right)=\left(\frac{1}{2}-\frac{1}{p}\right).$$
Alors $p$ borne la structure orbifolde de $\tilde E=\phi^*E$ et on a 
$$-\chi(\tilde E) \le \left(-2+\sum_1^{\tilde n}(1-\frac{1}{p})-(\tilde n-3)\right)=1-\frac{\tilde n}{p}.$$
Il vient 
$$d\left(\frac{1}{2}-\frac{1}{p}\right)\le 1-\frac{\tilde n}{p}$$
et donc
$$d\le2\frac{p-\tilde n}{p-2}.$$
Bien s\^ur, on veut $\tilde n>3$ et $d>1$ pour obtenir une d\'eformation non triviale ; la seule possibilit\'e reste donc $p=\infty$ et $d=2$. On v\'erifie ais\'ement qu'il n'y a pas de d\'eformation compl\`ete dans ce cas. Avec des arguments similaires, on exclut le cas $n=5$.
\end{proof}

\section{Solutions alg\'ebriques non \'el\'ementaires pour le syst\`eme de rang $2$}
Ici, on suppose que $E$ est hyperg\'eom\'etrique, uniformisante hyperbolique, avec $3$ p\^oles en $x=0,1,\infty$, que $\tilde E=\phi^*E$ a exactement $5$ singularit\'es non apparentes et que $$\phi:\mathbb P^1\to\mathbb P^1$$ de degr\'e $d=\deg(\phi)$ ramifie au dessus de $2$ points distincts des p\^oles de $E$.
 \begin{prop}\label{mod}
 Sous les hypoth\`eses pr\'ec\'edentes, notons $(p_0,p_1,p_\infty)$ la structure orbifolde de $E$, avec $2\le p_0\le p_1\le p_\infty\le\infty$ et $d:=\deg(\phi)$. Alors on a 
 \begin{equation}\label{mod2}
 d\left(1-\frac{1}{p_0}-\frac{1}{p_1}-\frac{1}{p_{\infty}}\right)\le 1-\frac{5}{p_{\infty}}.
 \end{equation}
\end{prop}
\begin{proof}
Si on note $(\tilde p_1,\tilde p_2,\tilde p_3,\tilde p_4,\tilde p_5)$ la structure orbifolde sous-jacente de $\tilde E$, avec $2\le\tilde p_1\le\tilde p_2\le\tilde p_3\le\tilde p_4\le\tilde p_5\le\infty$, alors on a $\tilde p_5\le p_\infty$ : les singularit\'es non apparentes en haut proviennent de celles d'en bas ; leurs exposants sont des multiples de ceux d'en bas. On a ainsi
$$-\chi(\tilde E)\le 2\tilde g-2+\sum_{i=1}^5(1-\frac{1}{\tilde{p_i}})-2=1-\frac{1}{\tilde{p_1}}-\frac{1}{\tilde{p_2}}-\frac{1}{\tilde{p_3}}-\frac{1}{\tilde{p_4}}-\frac{1}{\tilde{p_5}}\le1-\frac{5}{p_\infty}.$$
Par ailleurs
$$d\left(1-\frac{1}{p_0}-\frac{1}{p_1}-\frac{1}{p_\infty}\right)=-d\cdot\chi(E)\le-\chi(\tilde E)$$
d'o\`u l'in\'egalit\'e (\ref{mod2}). 
\end{proof}
\begin{prop}\label{ty}
On est toujours sous les hypoth\`eses pr\'ec\'edentes.
 \begin{enumerate}
 \item On a l'\'egalit\'e suivante : 
 \begin{equation}\label{mod1}
d-\lfloor\frac{d}{p_0}\rfloor-\lfloor\frac{d}{p_1}\rfloor-\lfloor\frac{d}{p_{\infty}}\rfloor=1
\end{equation}
o\`u $\lfloor\rfloor$ d\'esigne la partie enti\`ere.
\item Si $p_0\le  d<p_1\le p_{\infty}$ alors 
 \begin{equation}\label{mod3}\frac{1}{p_0}+\frac{1}{d}\ge 1.\end{equation}
 \item Si $p_0\le p_1\le d<p_{\infty}$ alors 
 \begin{equation}\label{mod3bis}\frac{1}{p_0}+\frac{1}{p_1}+\frac{1}{d}\ge 1.\end{equation}
\item Si $p_{\infty}\le d$ alors 
\begin{equation}\label{mod4} \frac{4}{5}\le\frac{1}{p_0}+\frac{1}{p_{1}}<1. \end{equation}
\end{enumerate}
\end{prop}
\begin{proof}D'apr\`es Riemann-Hurwitz, $\phi$ poss\`ede $2d-2$ points de ramification compt\'es avec multiplicit\'e dont $2$ au moins s'envoient en dehors de $x=0,1,\infty$. Le nombre de points au dessus de $x=0,1,\infty$ est donc minor\'e par 
$$\#\phi^{-1}(\{0,1,\infty\})\ge 3d-(2d-2)+2=d+4.$$
Le nombre de singularit\'es apparentes de $\tilde E$ au dessus de $x=i$ est major\'e par $\lfloor\frac{d}{p_i}\rfloor$, $i=0,1,\infty$. Ainsi, le nombre de singularit\'es non apparentes de $\tilde E$ (c'est \`a dire $5$) est minor\'e par 
$$5\ge d+4-\lfloor\frac{d}{p_0}\rfloor-\lfloor\frac{d}{p_1}\rfloor-\lfloor\frac{d}{p_{\infty}}\rfloor.$$
D'un autre c\^ot\'e, $\lfloor\frac{d}{p_i}\rfloor\le\frac{d}{p_i}$ et on a
$$d+4-\lfloor\frac{d}{p_0}\rfloor-\lfloor\frac{d}{p_1}\rfloor-\lfloor\frac{d}{p_{\infty}}\rfloor\ge 
d+4-\frac{d}{p_0}-\frac{d}{p_1}-\frac{d}{p_{\infty}}\ge 4+d(1-\frac{1}{p_0}-\frac{1}{p_1}-\frac{1}{p_{\infty}})>4$$
($E$ est hyperbolique) et puisque le terme de gauche est entier, on obtient 
$$d+4-\lfloor\frac{d}{p_0}\rfloor-\lfloor\frac{d}{p_1}\rfloor-\lfloor\frac{d}{p_{\infty}}\rfloor\ge 5$$
d'o\`u l'\'egalit\'e (\ref{mod1}).
Si $p_{\infty}>d$, on a $\lfloor\frac{d}{p_{\infty}}\rfloor=0$ et l'\'equation (\ref{mod1}) devient $d-\lfloor\frac{d}{p_0}\rfloor-\lfloor\frac{d}{p_1}\rfloor=1$. On sait que $d(1-\frac{1}{p_0}-\frac{1}{p_1})\le d-\lfloor\frac{d}{p_0}\rfloor-\lfloor\frac{d}{p_1}\rfloor=1$ alors on a $\frac{1}{p_0}+\frac{1}{p_1}+\frac{1}{d}\ge 1$.

Si $p_{\infty}\le d$, d'apr\`es l'hyperbolicit\'e de $E$, on a $\frac{1}{p_0}+\frac{1}{p_{1}}<1$. On remplace $d$ par $p_{\infty}$ dans l'in\'egalit\'e (\ref{mod2}) et on obtient l'in\'egalit\'e $(1-\frac{1}{p_0}-\frac{1}{p_1})p^2_{\infty}-2p_{\infty}+5\le 0$.  Pour qu'elle admette une solution, son discriminant doit \^etre sup\'erieur ou \'egal \`a $0$. 
Ainsi on a $\frac{1}{p_0}+\frac{1}{p_1}\ge \frac{4}{5}$.  
\end{proof}

\begin{rema}De la preuve de l'\'egalit\'e (\ref{mod1}), on d\'eduit que $\phi$ r\'ealise n\'ecessairement le nombre maximum possible de points singuliers apparents pour $\tilde E$ au dessus de $x=0,1,\infty$. Autrement dit, au dessus de $x=0$ par exemple, $\phi$ a exactement $\lfloor\frac{d}{p_0}\rfloor$ points ramifiant \`a l'ordre $p_0-1$ (d'indice $p_0$).
\end{rema}
Examinons d'abord le cas $d<p_{\infty}$. Les in\'egalit\'es  (\ref{mod3}) et (\ref{mod3bis}) nous donnent les possibilit\'es suivantes :
\begin{center}
\begin{tabular}{|c|c|}
\hline 
Triplets $(p_0, p_1, p_{\infty})$& Degr\'es de rev\^etement $d$\\
\hline
$(2, p_1, p_{\infty})$& $2$\\
\hline
$(2,3, p_{\infty})$& $3$, $4$, $5$, $6$\\
\hline
$(2,4, p_{\infty})$ & $4$\\
\hline
$(3,3, p_{\infty})$ & $3$\\
\hline
\end{tabular}
\end{center}
Les param\`etres $p_1$ et $p_{\infty}$ du tableau sont des entiers naturels suffisamment grands pour  v\'erifier la condition d'hyperbolicit\'e, par hypoth\`ese $>d$, ou infinis. Puisque ces param\`etres n'interviennent pas dans les contraintes de la construction, on peut tout aussi bien les fixer \`a $\infty$, ce que nous ferons ensuite.
Dans chacun des cas, on v\'erifie si la condition (\ref{mod1}) est satisfaite ; c'est toujours le cas
sauf pour le rev\^etement de degr\'e $5$. En tenant compte du fait que $\phi$ doit avoir le nombre maximal de points singuliers apparents (voir remarque pr\'ec\'edente), on trouve la liste suivante
\begin{table}[htdp]
\begin{center}
\begin{tabular}{|c|c|c|c|}
\hline 
Triplets $(p_0, p_1,p_{\infty})$& Degr\'es $d$& Type de ramifications & $N$ \\
\hline
$(2,\infty,\infty)$ & $2$ & $(2 ; 1+1 ; 1+1)$ & $1$ \\
\hline
 & $3$ & $(2+1 ; 3 ; 1+1+1)$ & $1$ \\
%\hline
$(2,3,\infty)$ & $4$ & $(2+2 ; 3+1 ; 1+1+1+1)$ & $2$ \\
%\hline
 & $6$ & $(2+2+2 ; 3+3 ; 1+1+1+1+1+1)$ & $3$ \\
\hline 
$(2,4,\infty)$ & $4$ & $(2+2 ; 4 ; 1+1+1+1)$ & $1$ \\
\hline
$(3,3,\infty)$ & $3$ & $(3 ; 3 ; 1+1+1)$ & $0$ \\
\hline
\end{tabular}
\end{center}
\end{table}%

Par exemple, pour la ligne $3$, les ramifications au dessus de $x=0,1$ sont impos\'ees et il nous reste $2$ points de ramification que l'on peut choisir fixe au dessus de $x=\infty$ ou libre ; dans chacun des cas,  la condition (\ref{mod1}) nous assure que la d\'eformation 
est compl\`ete. On a syst\'ematiquement privil\'egi\'e de maximiser le nombre de points libres dans le tableau, les autres cas de figure s'en d\'eduisent facilement. Par exemple, toujours pour la ligne $3$, on a le cas d\'eg\'en\'er\'e de type Painlev\'e VI $(2+2 ; 3+1 ; 2+1+1)$.
Seules les lignes $3$ et $4$ nous fournissent des solutions possibles pour les syst\`emes de Garnier avec $N>1$.
Examinons maintenant le cas $p_0\le p_1\le p_\infty\le d$. L'in\'egalit\'e (\ref{mod4}) nous dit que $(p_0,p_1)=(2,3)$ puis en utilisant (\ref{mod2}) et le fait que $p_\infty\le d$, on trouve la liste suivante
\begin{center}
\begin{tabular}{|c|c|}
\hline
Triplets $(p_0, p_1, p_{\infty})$& Degr\'es de rev\^etement $d$\\
\hline
$(2,3,7)$ & $7$, $8$, $9$, $10$, $11$, $12$\\
\hline
$(2, 3, 8)$ & $8$, $9$\\
\hline
\end{tabular}
\end{center}
Seul le rev\^etement de degr\'e $11$ ne satisfait pas l'\'egalit\'e (\ref{mod1}). Pour les autres, on trouve
\begin{table}[htdp]
\begin{center}
\begin{tabular}{|c|c|c|c|}
\hline 
Triplets $(p_0, p_1,p_{\infty})$& Degr\'es $d$& Type de ramifications & $N$ \\
\hline
 & $7$ & $(2+2+2+1 ; 3+3+1 ; 7)$ & $-1$ \\
%\hline
 & $8$ & $(2+2+2+2 ; 3+3+1+1 ; 7+1)$ & $0$ \\
%\hline
$(2,3,7)$ & $9$ & $(2+2+2+2+1 ; 3+3+3 ; 7+1+1)$ & $0$ \\
%\hline
 & $10$ & $(2+2+2+2+2 ; 3+3+3+1 ; 7+1+1+1)$ & $1$ \\
%\hline
 & $12$ & $(\underbrace{2+\cdots+2 }_{6 fois} ; 3+3+3+3 ; 7+\underbrace{1+\cdots+1 }_{5 fois})$ & $2$ \\
 \hline 
$(2,3,8)$ & $8$ & $(2+2+2+2 ; 3+3+1+1 ; 8)$ & $-1$ \\
\hline
 & $9$ & $(2+2+2+2+1 ; 3+3+3 ; 8+1)$ & $-1$ \\
\hline
\end{tabular}
\end{center}
\end{table}%
\vspace{4cm}

Les lignes avec $N=-1$ correspondent \`a $\tilde n=2$ singularit\'es en haut : il n'y a certainement pas de  rev\^etement correspondant. Les lignes $N=0$, lorsque le rev\^etement existe, vont nous donner de nouveau une \'equation hyperg\'eom\'etrique en haut. La ligne $N=1$ nous donne une des solutions alg\'ebriques de  Painlev\'e VI dans la liste de Doran. La seule possibilit\'e pour nous est la ligne $5$ avec $N=2$. On a d\'emontr\'e le
\begin{theo}\label{Inver}
Si $E$ est une \'equation hyperg\'eom\'etrique non \'el\'ementaire dont les monodromies locales sont d'ordre $(p_0, p_1,p_{\infty})$
et si $\phi:\mathbb P^1\to\mathbb P^1$ est une famille \`a $2$ param\`etres de rev\^etements ramifi\'es tel que $\phi_t^*E$ a $5$ singularit\'es non apparentes, alors on est dans la liste  :
\begin{table}[htdp]
\begin{center}
\begin{tabular}{|c|c|c|}
\hline 
Triplets $(p_0, p_1,p_{\infty})$& Degr\'es $d$& Type de ramifications\\
\hline
$(2,3,\infty)$ & $4$ & $(2+2 ; 3+1 ; 1+1+1+1)$\\
\hline
$(2,3,\infty)$ & $6$ & $(2+2+2 ; 3+3 ; 2+1+1+1+1)$\\
\hline 
$(2,3,7)$ &$12$ &   $(\underbrace{2+\cdots+2 }_{6 fois}; 3+3+3+3; 7+1+1+1+1+1)$ \\
\hline
\end{tabular}
\end{center}
\caption{Listes Pull-back hyperg\'eom\'etriques}
\label{H}
\end{table}%
\end{theo}

Dans le tableau,  $(p_0, p_1,p_{\infty})$ d\'esigne l'ordre local des monodromies de l'hyperg\'eom\'etrique $E$ aux point singuliers $0$, $1$ et $\infty$, la seconde colonne donne le degr\'e de $\phi_t$ et la colonne de droite, la partition des fibres de $\phi_t$ au dessus de $0$, $1$ et $\infty$.
\begin{rema}  
Dans ce th\'eor\`eme, on constate qu'il n'existe pas de pull-back \`a $2$ param\`etres libres de triangles hyperboliques $(2,\infty, \infty)$, $(2,3,8)$, $(2,4,p_{\infty})$ et $(3, 3,p_{\infty})$ o\`u $p_{\infty}$ est un entier naturel tr\`es grand. 
Les triangles hyperboliques $(2,3, \infty)$ et $(2,3,8)$ admettent des pull-back \`a $1$ param\`etre de degr\'es $3$ et $9$ respectivement ; on a
\begin{description}
\item[Pour $d=3$] $(2,3, \infty)$; $(2+1; 1+1+1; 3)$
 \item[Pour $d=9$] $(2,3,8)$ ; $(2+2+2+2+1 ; 3+3+1+1+1; 8+1)$.
 \end{description}
Ceci construit des solutions alg\'ebriques incompl\`etes.
\end{rema}
\subsection{Existence des rev\^etements}
Le tableau des ramifications \ref{H} satisfait la formule de Riemann-Hurwitz, ce qui ne suffit pas \`a montrer l'existence des rev\^etements ramifi\'es correspondant. Dans la derni\`ere section, on donnera des expressions explicites pour le rev\^etement concernant la lignes 1 du tableau.
Pour la 3ième ligne, les calculs sont trop compliqu\'es ; on d\'emontre l'existence comme suit.

Il faut construire le rev\^etement topologique $\phi:\mathbb  S^2\to \mathbb P^1$ : il existera alors une unique structure complexe sur la sph\`ere $\mathbb S^2$ rendant la fl\^eche holomorphe, et donc rationnelle. L'existence d'un rev\^etement topologique se d\'emontre en construisant une repr\'esentation 
$$\pi_1(\mathbb P^1\setminus\{\text{valeurs critiques}\})\to\mathrm{Perm}\{1,\ldots,d\}$$
o\`u $d$ est le degr\'e du rev\^etement : on le construit via sa monodromie.
En fait, le tableau de ramification compl\`ete est 
$$ (\underbrace{2+\cdots+2 }_{6 fois}\ ;\ 3+3+3+3 ; 7+1+1+1+1+1 \ ;\  \underbrace{1+\cdots+1 }_{10 fois}+2 \ ;\ \underbrace{1+\cdots+1 }_{10 fois}+2)$$
au dessus des points critiques
$$(0,1,\infty,\lambda_1,\lambda_2).$$
Il faut donc trouver des permutations $\sigma_0,\sigma_1,\sigma_\infty,\tau_1,\tau_2\in\mathrm{Perm}\{1,\ldots,d\}$
conjugu\'ees individuellement \`a 
$$\sigma_0\sim(12)(34)(56)(78)(9\ 10)(11\ 12),\ \ \ \sigma_1\sim(123)(456)(789)(10\ 11\ 12),\ \ \ \sigma_\infty\sim(1234567)$$
$$\tau_1\sim\tau_2\sim (12)\ \ \ \text{avec}\ \ \ \sigma_0\circ\sigma_1\circ\sigma_\infty\circ\tau_1\circ\tau_2=\text{identit\'e}.$$
En fait, on trouve chez Doran et Kitaev une famille \`a $1$ param\`etre de rev\^etements ramifi\'es du type
$$ (\underbrace{2+\cdots+2 }_{6 fois}\ ;\ 3+3+3+3 ; 7+2+1+1+1 \ ;\  \underbrace{1+\cdots+1 }_{10 fois}+2)$$
C'est une sous-famille de celle que nous cherchons correspondant \`a une confluence : une des deux ramifications
libres, disons $\lambda_1$, est \`a l'infini. Correspond alors un rev\^etement topologique et une repr\'esentation
donn\'ee par 
$$\sigma_0'\sim(12)(34)(56)(78)(9\ 10)(11\ 12),\ \ \ \sigma_1'\sim(123)(456)(789)(10\ 11\ 12),\ \ \ \sigma_\infty'\sim(1234567)(89)$$
$$\tau_2'\sim (12)\ \ \ \text{avec}\ \ \ \sigma_0'\circ\sigma_1'\circ\sigma_\infty'\circ\tau_2'=\text{identit\'e}.$$
Puisque $(1234567)(89)=(1234567)\circ(89)$, il est clair que l'on peut d\'ecomposer $\sigma_\infty'=\sigma_\infty\circ\tau_1$, ce qui montre l'existence du rev\^etement.

\subsection{Exposants $(\tilde{\theta}_{0}, \tilde{\theta}_{1}, \tilde{\theta}_{t_1}, \tilde{\theta}_{t_2}, \tilde{\theta}_{\infty})$}
On a travaill\'e jusque maintenant avec la structure orbifolde sous-jacente à l'\'equation $E$. Pour chaque type
de rev\^etement list\'e dans le th\'eor\`eme \ref{Inver}, on peut choisir pour $E$ toute \'equation dont la structure
orbifolde sous-jacente est donn\'ee par $(p_0, p_1,p_{\infty})$. 

Par exemple, si on consid\`ere la premi\`ere ligne du tableau \ref{H}, en prenant pour $E$ l'\'equation hyperg\'eom\'etrique d'exposants $(\frac{1}{2},\frac{1}{3},\theta_\infty)$, $\theta_\infty\in\mathbb C$ arbitraire, on aura apr\`es
pull-back une famille isomonodromique d'\'equations fuchsiennes avec $5$ p\^oles non apparents d'exposants
$$(\frac{1}{3},\theta_\infty,\theta_\infty,\theta_\infty,\theta_\infty)$$
et deux p\^oles apparents d'exposant $2$.
Si, par contre, on part de l'\'equation hyperg\'eom\'etrique d'exposants $(\frac{1}{2},\frac{2}{3},\theta_\infty)$,
on obtiendra alors une famille d'\'equations fuchsiennes avec $5$ p\^oles non apparents d'exposants
$$(\frac{2}{3},\theta_\infty,\theta_\infty,\theta_\infty,\theta_\infty)$$
et trois p\^oles apparents d'exposant $2$. Pourtant, les deux familles ainsi obtenues sont reli\'ees par une
\'equivalence de jauge birationnelle. En fait, on peut voir cette \'equivalence en bas, sur les hyperg\'eom\'etriques :
$$(\frac{1}{2},\frac{2}{3},\theta_\infty)=(1-\frac{1}{2},1-\frac{1}{3},\theta_\infty)\sim(\frac{1}{2},\frac{1}{3},\theta_\infty).$$
Cette \'equivalence de jauge se rel\`eve en une \'equivalence de jauge entre les deux familles isomonodromiques. Il nous faut donc classer les \'equations fuchsiennes pour chacune des $3$ orbifoldes uniformisantes donn\'ee par le tableau  \ref{H} modulo \'equivalence de jauge birationnelle. On trouve le tableau suivant :
\begin{table}[htdp]
\begin{center}
\begin{tabular}{|c|c|c|}
\hline 
Degr\'es & Exposants de $E$& Exposants de $E'$\\
\hline

$3$& $(\frac{1}{2}, \frac{1}{3},\theta_{\infty})$ & $(\frac{1}{2}, \frac{1}{3},\frac{1}{3},\frac{1}{3},3\theta_{\infty})$\\

\hline

$4$ & $(\frac{1}{2}, \frac{1}{3},\theta_{\infty})$ &$(\frac{1}{3},\theta_{\infty},\theta_{\infty},\theta_{\infty},\theta_{\infty})$\\

\hline

$6$ & $(\frac{1}{2}, \frac{1}{3},\theta_{\infty})$ &$(2\theta_{\infty},\theta_{\infty},\theta_{\infty},\theta_{\infty},\theta_{\infty})$\\

\hline

$12$& $(\frac{1}{2}, \frac{1}{3},\frac{1}{7})$ & $(\frac{1}{7}, \frac{1}{7},\frac{1}{7},\frac{1}{7},\frac{1}{7})$\\

%\hline
& $(\frac{1}{2}, \frac{1}{3},\frac{2}{7})$ & $(\frac{2}{7}, \frac{2}{7},\frac{2}{7},\frac{2}{7},\frac{2}{7})$\\

%\hline

& $(\frac{1}{2}, \frac{1}{3},\frac{3}{7})$ & $(\frac{3}{7}, \frac{3}{7},\frac{3}{7},\frac{3}{7},\frac{3}{7})$\\

\hline

$9$ & $(\frac{1}{2}, \frac{1}{3},\frac{1}{8})$&$(\frac{1}{2}, \frac{1}{3},\frac{1}{3},\frac{1}{3},\frac{1}{8})$\\

%\hline

 & $(\frac{1}{2}, \frac{1}{3},\frac{3}{8})$&$(\frac{1}{2}, \frac{1}{3},\frac{1}{3},\frac{1}{3},\frac{3}{8})$\\
\hline
\end{tabular}
\end{center}
\caption{Exposants de l'\'equation $\phi_t^*E$}
\label{Mono}
\end{table}

\section{Solutions alg\'ebriques du syst\`eme de Garnier d'ordre $3$}
On s'int\'eresse maintenant au cas $N=3$ c'est \`a dire l'application $\phi:\Bbb P^1_x\to \Bbb P^1_z$ a trois points critiques libres distincts en dehors sur des valeurs critiques $z=0$, $1$ et en $\infty$. Comme la section pr\'ec\'edente elle sera d\'efinie par une famille de rev\^etements param\'etr\'es par $3$ param\`etres. L'\'equation transform\'ee $E'=\phi^*E$ par $\phi$ est une \'equation fuchsienne avec $6$ points singuliers non apparents normalis\'es projectivement \`a $x=0,1,t_1,t_2,t_3,\infty$ et $3$ points singuliers apparents normalis\'es projectivement \`a $x=q_1,q_2, q_3$. Les six singularit\'es non apparentes de $E'$ sont des points critiques non apparents de $\phi$, toutes sont situ\'ees au dessus de l'ensemble $\{0,1,\infty\}$. Tous les points de ramifications de rev\^etement $\phi$ se trouvent au dessus de $z=0,1,\infty$ sauf $3$ points de ramifications simples distincts libres $x=q_1,q_2, q_3$. Alors, d'apr\`es la formule de Hurwitz, l'application $\phi$ contient $d+5$ points distincts sur l'ensemble $\{0,1,\infty\}$. Ces trois points critiques libres $x=q_1(t_1,t_2,t_3)$, $x=q_2(t_1,t_2,t_3)$ et $x=q_3(t_1,t_2,t_3)$, des fonctions alg\'ebriques en $t_1,t_2,t_3$, sont des solutions alg\'ebriques du syst\`eme de Garnier d'ordre $3$. On veut \'etudier la classification de ces solutions alg\'ebriques du syst\`eme de Garnier de rang $3$ provenant du pull-back $\phi$ de l'\'equation fuchsienne $E$ de monodromies locales $(\frac{1}{p_0},\frac{1}{p_1},\frac{1}{p_{\infty}})$. On commence par $E$ \'equation fuchsienne hyperg\'eom\'etrique, i.e $\frac{1}{p_0}+\frac{1}{p_1}+\frac{1}{p_{\infty}}<1$ et le groupe de monodromie de l'\'equation fuchsienne $E'$ est irr\'eductible. On a obtenu des r\'esultats similaires \`a la proposition (\ref{ty}) :

\begin{prop}\label{types}
Soit $\phi:\Bbb P^1_x\to \Bbb P^1_z$ une application rev\^etement de $\Bbb P^1_z$ de degr\'e $d$, $E$ une \'equation hyperg\'eom\'etrique fix\'ee sur $\Bbb P^1_z$ avec monodromie $(\frac{1}{p_0}, \frac{1}{p_1},\frac{1}{p_{\infty}})$ tel que $$\frac{1}{p_0}+\frac{1}{p_1}+\frac{1}{p_{\infty}}<1.$$ Supposons que l'\'equation fuchsienne obtenue $E'=\phi^*E$ par le pull-back de l'\'equation $E$ a $6$ points singuliers non apparents normalis\'es \`a $0$, $1$, $t_1$, $t_2$, $t_3$ et $\infty$ d'ordres de monodromies locales autour de chacun de ces points sont respectivement $\tilde{p}_0$, $\tilde{p}_1$, $\tilde{p}_{t_1}$, $\tilde{p}_{t_2}$, $\tilde{p}_{t_3}$ et $\tilde{p}_{\infty}$ et $3$ points singuliers apparents. Alors  \begin{enumerate}
\item on a l'\'egalit\'e $d-(\lfloor \frac{d}{p_0}\rfloor+ \lfloor \frac{d}{p_1}\rfloor+\lfloor \frac{d}{p_{\infty}}\rfloor)=1$;
\item on obtient l'in\'egalit\'e $(1-\frac{1}{p_0}-\frac{1}{p_1}-\frac{1}{p_{\infty}})d\le 1-\frac{6}{p_{\infty}}$;
\item si $p_{\infty}>d$ on a $\frac{1}{d}+\frac{1}{p_0}+\frac{1}{p_1}\ge 1$;
\item si $p_{\infty}\le d$ on obtient $(1-\frac{1}{p_0}-\frac{1}{p_1})p^2_{\infty}-2p_{\infty}+6\le 0$ et $\frac{5}{6}\le \frac{1}{p_0}+\frac{1}{p_1}<1$.
\end{enumerate}
\end{prop}
\begin{proof}
La d\'emonstration est le m\^eme raisonnement que celle de proposition (\ref{ty}), mais pour le premier \'enonc\'e il faut savoir que le nombre total des points distincts dans les trois fibres de $\phi$ est \'egal \`a $d+5$ et pour le deuxi\`eme \'enonc\'e il faut aussi tenir compte que l'\'equation fuchsienne $E'$ poss\`ede $6$ points singuliers non apparentes (points singuliers essentiels). Les deux derniers \'enonc\'es sont les d\'eductions des deux premiers. 
\end{proof}
On constate, dans la proposition (\ref{types}), que tous les triplets hyperboliques $(p_0,p_1,p_{\infty})$ ne satisfont pas la condition du quatri\`eme \'enonc\'e de proposition. Cela veut dire que $d<p_{\infty}$ alors les points sur $\infty$ sont des points singuliers non apparents de l'\'equation fuchsienne $E'$ et les ordres de monodromie autour de ces points peuvent \^etre \'egaux \`a l'infini. On obtient la liste des triplets $(p_0,p_1,p_{\infty})$ et ses degr\'es $d$ satisfaisant la troisi\`eme condition : \begin{itemize}
\item $(2,3,p_{\infty})$ $d=2, 3, 4, 5, 6$;
\item $(2,4,p_{\infty})$ $d=2,3,4$;
\item $(2,p_1,p_{\infty})$ $d=2$;
\item $(3,3,p_{\infty})$ $d=2,3$.
\end{itemize}
Si on applique le premier \'enonc\'e de la proposition (\ref{types}) certains degr\'es $d$ de l'application $\phi$ disparaissent. On obtient \`a la fin la liste des triplets $(p_0,p_1,p_{\infty})$ et de degr\'es $d$ satisfaisant les \'enonc\'es de la proposition (\ref{types}) :

\begin{table}[htdp]
\begin{center}
\begin{tabular}{|c|c|}
\hline
Triplets $(p_0,p_1,p_{\infty})$ &Degr\'es $d$\\
\hline
$(2,3,p_{\infty})$ & $2, 3, 4,6$\\
\hline
 $(2,4,p_{\infty})$&$2,4$\\
 \hline
 $(2,p_1,p_{\infty})$ & $2$\\
 \hline
 $(3,3,p_{\infty})$ & $3$\\
 \hline
\end{tabular}
\end{center}
\caption{Triplets et Rev\^etements correspondants}
\label{gar3}
\end{table}
%\vspace{2.4cm}
On cherche maintenant des rev\^etements qui transforment l'\'equation hyperg\'eom\'etrique $E$ \`a l'\'equation fuchsienne avec $6+3$ p\^oles simples. Ils seront choisi parmi les \'el\'ements pr\'eliminairement selectionn\'es dans le tableau (\ref{gar3}), on a eu le r\'esultat suivant.
\begin{theo}
Les solutions alg\'ebriques du syst\`eme de Garnier d'ordre $3$ avec groupe de monodromie irr\'eductible infini (Zariski dense) sont construites par la d\'eformation isomonodromique de l'\'equation fuchsienne obtenue $E'$ par le pull-back $\phi$ de l'\'equation hyperg\'eom\'etrique $E$. Alors il existe un et un seul rev\^etement $\phi$ \`a homographie pr\`es  de degr\'e $d$  et de type de ramification list\'e dans le tableau : 
\begin{table}[htdp]
\begin{center}
\begin{tabular}{|c|c|c|c|}
\hline
\multicolumn{2}{|c|}{\textbf{Monodromies}}&\multicolumn{1}{|c|} {Degr\'e de $\phi$} &\multicolumn{1}{|c|}{Type de rev\^etement}\\
\cline{1-2} 
\'Equation $E$&\'Equation $E'$ &$d$ &$(\cdots;\cdots;\cdots)$ \\ 
\hline
$(\frac{1}{2},\frac{1}{3},\theta)$&$(\theta,\theta,\theta,\theta,\theta,\theta)$&$6$&$(2+2+2; 3+3; \underbrace{1+\cdots+1}_{6 fois})$\\
\hline
\end{tabular}
\end{center}
\caption{Transformation de l'\'equation hyperg\'eom\'etrique \`a l'\'equation fuchsienne avec $6+3$ singuliers.}
\label{fig2}
\end{table}
%\vspace{2.5cm}
\end{theo}
\begin{proof}
D'apr\`es la proposition (\ref{types}) pour tous les triplets $(p_0,p_1,p_{\infty})$ fix\'es sur $\Bbb P^1_z$, on a $d<p_{\infty}$. On a l'in\'egalit\'e $\frac{d-s_0}{p_0}+\frac{d-s_1}{p_1}+\frac{d-s_{\infty}}{p_{\infty}}\ge d-1$ due \`a la formule de Riemann-Hurwitz \cite{Belyi} o\`u $s_0$, $s_1$ et $s_{\infty}$ sont le nombre de points critiques non apparents sur $i=0$, $1$ et $\infty$ respectivement telle que $\sum_{i=0,1,\infty}s_i\ge 6$. Si on fixe l'\'equation fuchsienne sur $\Bbb P^1_z$ avec  triplet $(2,3,p_{\infty})$. Dans le tableau (\ref{gar3}) le degr\'e maximal correspondant \`a ce triplet est $6$ et on peut supposer que $p_{\infty}\ge 7$ alors l'in\'egalit\'e ci-dessus devient $$2\le d\le 6-15s_0-8s_1.$$ On a une seule possibilit\'e sur $s_0$ et sur $s_1$ pour trouver une solution de l'in\'egalit\'e
\begin{center}  $s_0=s_1=0$ et $s_{\infty}\ge 6 \Rightarrow d\le 6$.
 \end{center} 
Cela veut dire que tous les points critiques non-apparents sont sur $\infty$ et les points sur $0$ et sur $1$ sont tous critiques apparents alors le degr\'e $d$ doit \^etre multiple commun de $2$ et $3$. On a une seule valeur de $d\in [2,6]$, que $2$ et $3$ lui divisent, est $d=6$. On a une seule fa\c con de partition de $d$ au dessus de $0$ et de $1$ pour trouver le nombre maximal de points critiques apparents. On a sur $0$ trois points critiques apparents et sur $1$ deux points critiques apparents. Donc on obtient le type de ramification de $\phi$ est donn\'e par $(2+2+2; 3+3; 1+1+1+1+1+1)$.\\
Si on consid\`ere l'\'equation fuchsienne $E$ correspondante au triplet $(2,4,p_{\infty})$ avec $p_{\infty}\ge 5$; puisque dans le tableau (\ref{gar3}) le degr\'e maximal correspondant au triplet $(2,4,p_{\infty})$ est  $d=4$ ; on a l'in\'egalit\'e $$d\le -4-6s_0-s_1.$$ Alors $d$ est major\'e par un nombre entier n\'egatif contredit du fait que $2\le d<p_{\infty}$. Donc il n'existe pas de rev\^etement $\phi$ qui tire en arri\`ere l'\'equation fuchsienne hyperg\'eom\'etrique $E$, avec param\`etre $(\frac{1}{2},\frac{1}{4},\frac{1}{p_{\infty}})$, \`a l'\'equation fuchsienne $E'$.\\
Pour le triplet $(2,p,p)$, il est associ\'e \`a une seule application de degr\'e $2$. Alors quel que soit la r\'epartition de $2$, le nombre des points distincts dans les trois fibres ne vaut pas $7$.\\
Pour $(3,3,p_{\infty})$ o\`u $p_{\infty}\ge 4$ ; pour toutes les valeurs de $s_0$ et $s_1$ on a $d\le -6-s_0-s_1$ cela contredit aussi $2\le d<p_{\infty}$. 
\end{proof}

\section{Syst\`eme de Garnier d'ordre sup\'erieur ou \'egal $4$} On suppose que l'\'equation fuchsienne obtenue $E'$ apr\`es la transformation de $E$ par $\phi$ poss\`ede $n\ge 7$ points singuliers non apparents et $N\ge 4$ singularit\'es apparentes. Si on suppose de plus que l'\'equation fuchsienne $E$ est hyperg\'eom\'etrique,  le deuxi\`eme et le quatri\`eme
\'enonc\'es de la proposition (\ref{types}) modifient respectivement en\begin{enumerate}
 \item $(1-\sum_{i=0,1,\infty}\frac{1}{p_i})d\le 1-\frac{n}{p_{\infty}}$,
\item   $p_{\infty}\le d$ on a $(1-\frac{1}{p_0}-\frac{1}{p_1})p^2_{\infty}-2p_{\infty}+n\le 0$ et $\frac{n-1}{n}\le \frac{1}{p_0}+\frac{1}{p_1}<1$. 
\end{enumerate}
On remarque dans la section pr\'ec\'edente qu'\`a partir de $n=6$ le degr\'e de rev\^etements existant est strictement inf\'erieur \`a $p_{\infty}$. Les \'el\'ements list\'es dans le tableau (\ref{gar3}) restent les m\^emes, aucun rev\^etement ne peut transformer l'\'equation $E$ en \'equation fuchsienne $E'$. 
\begin{theo}
 Si $n\ge 7$ il n'existe pas une solution alg\'ebrique compl\`ete du syst\`eme de Garnier, avec groupe de monodromies irr\'eductibles, obtenue par la m\'ethode de Kitaev.
\end{theo}
\begin{proof}
D'apr\`es la formule de Riemann-Hurwitz, on a $\frac{d-s_0}{p_0}+\frac{d-s_1}{p_1}+\frac{d-s_{\infty}}{p_{\infty}}\ge d-1$, o\`u $s_0$, $s_1$ et $s_{\infty}$ sont respectivement le nombre de points critiques non apparents sur $0$, $1$ et $\infty$ tel que $\sum_{i=0,1,\infty}s_i=n$. Si on fixe le triplet $(2,3,p)$, on obtient l'in\'egalit\'e ci-dessus sous la forme \begin{equation}\label{ineg}
d\le 6(7-n)-15s_0-8s_1.
\end{equation} 
Si $3\le n\le 6$ on obtient des solutions alg\'ebriques aux d\'eformations de deux \'equations fuchsiennes d\'ej\`a connues et aux deux autres \'equations que nous avons fait en haut. Si $n\ge 7$, le degr\'e $d$ du rev\^etement $\phi$ dans l'in\'egalit\'e (\ref{ineg}) est major\'ee par des entiers naturels n\'egatifs, alors $\phi$ n'existe pas. Cela entra\^ine la non existence de l'\'equation fuchsienne $E'$.\\ On voit qu'on ne peut plus trouver un rev\^etement $\phi$ d\`es que $n> 6$ et on peut tester les autres types :\begin{description}
\item[triplet $(2,4,p)$]  on a l'in\'egalit\'e $d\le 4(5-n)-6s_0-s_1$ si $n\ge 6$;
\item[ le triplet $(3,3,p)$] on obtient $d\le 3(4-n)-s_0-s_1$ si $n\ge 5$.
\end{description}
\end{proof}
\section{Exemple : Rev\^etement degr\'e $4$}
 On veut calculer explicitement le rev\^etement $\phi:\mathbb P^1_x\to\mathbb P^1_z$ de degr\'e $4$ du tableau (\ref{H}).  \`A homographie pr\`es en $z$, le type de ramification recherché est $(2+2; 1+1+1+1;3+1)$ : 
 \begin{itemize}
 \item la fibre $\phi^{-1}(0)$ est totalement ramifi\'ee \`a l'ordre $2$,
 \item la fibre $\phi^{-1}(1)$ est non ramifi\'ee, constitu\'ee des points $x=0,1,t_1,t_2$,
 \item la fibre $\phi^{-1}(\infty)$ a  un point simple $x=\infty$ et un point triple,
 \item $\phi$ a deux autres points critiques $x=q_1,q_2$.
 \end{itemize}
Ces contraintes nous imposent que l'application $\phi$ est de la forme 
$$\phi(x) = -\frac{c^3}{a_0^2}\frac{(x^2 + a_1x + a_0)^2}{(x-c)^3}$$
o\`u les constantes $a_0,a_1,c$ satisfont ($\phi(1)=1$)
$$\frac{a_0^2}{c^3}+\frac{(1 + a_1 + a_0)^2}{(1-c)^3}.$$
Cette \'equation d\'efinit une surface rationnelle $S$ que l'on peut param\'etrer par $(a_0,s)$ :
$$a_1 = a_0(s^3-1)-1\ \ \ \text{et}\ \ \ c =\frac{1}{1-s^2}.$$
Les deux points $x=t_1,t_2$ sont alors les racines de l'\'equation polyn\^omiale
$$t^2 - (a_0^2(s^2-1)^3-2 a_0(s^3-1)+1)t +a_0(2-a_0(2s^3-3s^2+1))=0$$
qui d\'efinit une nouvelle surface $S_t$ : la projection
$S_t\to S;(t,a_0,s)\mapsto (a_0,s)$ est un rev\^etement double ramifiant le long du discriminant du polyn\^ome pr\'ec\'edent.
C'est encore une surface rationnelle que l'on peut param\'etrer par $(s,t)$ en posant
$$a_0=\frac{t^2-1}{(s+1)((s-1)^2t^2-(s+1)^2)}.$$
On obtient alors les solutions $x=t_1,t_2$ :
$$t_1 =\frac{ -(t+1)(st-t-1-3s)}{(st-t-1-s)^2(s+1)}\ \ \ \text{et}\ \ \ 
t_2 = -\frac{(-1+t)(st-t+1+3s)}{(st-t+s+1)^2(s+1)}.$$
On v\'erifie sans peine que l'application 
$$\Pi_t:S_t\to\mathbb P^1\times\mathbb P^1\ ;\ (s,t)\mapsto(t_1,t_2)$$
est de degr\'e $8$, ramifiant pr\'ecis\'ement au dessus de 
$$t_1=0,1,\infty, \ \ \ t_2=0,1,\infty \ \ \ \text{et la diagonale}\ \ \ t_1=t_2.$$
C'est la propri\'et\'e de Painlev\'e. 
Il nous reste \`a d\'eterminer les deux points critiques libres $x=q_1,q_2$. Ils sont solutions 
de l'\'equation polyn\^omiale 
$$x^2-\frac{t^2s^3+3s^3-4t^2s^2+4s^2+5t^2s+7s+2-2t^2}{(s-1)(s+1)(st-t+s+1)(st-t-1-s)}x$$
$$-\frac{(st-t+1+3s)(st-t-1-3s)}{(s+1)^2(st-t+s+1)(st-t-1-s)(s-1)}=0$$
dont  le discriminant est donn\'e par $s^2(s+1)^2F(s,t)$ avec 
$$F(s,t)=s^4t^4+6s^4t^2+9s^4-4t^4s^3-56t^2s^3+60s^3+6t^4s^2+100t^2s^2$$
$$+118s^2-4t^4s-56t^2s+60s+t^4+6t^2+9.$$ 
Pour d\'eterminer les racines $x=q_1,q_2$, il nous faut passer de nouveau \`a un rev\^etement double $S_q\to S_t$
ramifiant le long de la courbe alg\'ebrique $F(s,t)=0$ (afin que le discriminant devienne un carr\'e).
Cette courbe est irr\'eductible sur $\Bbb Q$ mais r\'eductible sur $\Bbb Q(\alpha)$, avec $\alpha^2+3=0$. On obtient que la courbe $F(s,t)=0$ est la r\'eunion de deux coniques $F_1$ et $F_2$ qui s'intersectent en quatre points o\`u \begin{eqnarray*}
F_1(s,t)&=&t^2+2t\alpha-3-2t^2s-4st\alpha-10s+t^2s^2+2s^2t\alpha-3s^2\\
F_2(s,t)&=&t^2s^2-2s^2t\alpha-3s^2-2t^2s+4st\alpha-10s+t^2-2t\alpha-3. 
\end{eqnarray*} 
En param\'etrant le pinceau de coniques associ\'es, on obtient une param\'etrisation rationnelle $(u,v)$ pour $S_q$ en posant
$$s=-\frac{4u(v'u+2)}{u^2-3-uv'},\ \ \  t=-\frac{3u^2+v'u-1}{u^2-3-v'u}\ \ \ \text{avec}\ \ \ \frac{-2\alpha+v'}{2\alpha+v'}=v^2.$$
On obtient enfin la param\'etrisation compl\`ete de la solution alg\'ebrique (rationnelle sur $\Bbb Q(\alpha)$) :
\begin{eqnarray*}
t_1(u,v)& = &-{1\over52}(353+9\alpha)(-4\alpha u^2+2uv^2-2u-v^2+1+13u^2v^2+11u^2+4\alpha uv^2 \\
& & -4\alpha u-2\alpha v^2+2\alpha)u(2v-1+\alpha)(2v+1-\alpha)(uv+u+\alpha v-\alpha)^2(uv-u\\
& &+\alpha v+\alpha)^2/((u+1)(u^2v^2-u^2-2\alpha uv^2-2\alpha u+v^2-1)(v+1)(v-1)(\alpha v\\
& & -2v+\alpha-2+7uv+u-4\alpha u)^2(\alpha v-2v+2-\alpha+7uv-u+4\alpha u)^2);\\
%\end{eqnarray*}
%\begin{eqnarray*}
t_2(u,v)&=&{1\over52}(9\alpha-353)(-v^2+1+13u^2v^2+11u^2+4\alpha u^2-2uv^2+2u+2\alpha v^2\\
& &-2\alpha+4\alpha uv^2-4\alpha u)u(2v-1-\alpha)(2v+1+\alpha)(uv+u+\alpha v-\alpha)^2(uv\\
& &-u+\alpha v+\alpha)^2/((u-1)(u^2v^2-u^2-2\alpha uv^2-2\alpha u+v^2-1)(v+1)(v\\
& &-1)(2v+\alpha v-2-\alpha+7uv-u-4\alpha u)^2(2v+\alpha v+2+\alpha\\
& &+7uv+u+4\alpha u)^2);\\
q_1(u,v) &=& -{7\over 2}(uv-u+\alpha v+\alpha)u(2v+1-\alpha)(2v+1+\alpha)(uv+u+\alpha v\\
& & -\alpha)^2/((u^2v^2-u^2-2\alpha uv^2-2\alpha u+v^2-1)(\alpha v-2v+2-\alpha+7uv-u\\
& &+4\alpha u)(2v+\alpha v-2-\alpha+7uv-u-4\alpha u)(v+1));\\
q_2(u,v)& =& -{7\over 2}(uv+u+\alpha v-\alpha)u(2v-1+\alpha)(2v-1-\alpha)(uv-u+\alpha v\\
& &+\alpha)^2/((u^2v^2-u^2-2\alpha uv^2-2\alpha u+v^2-1)(2v+\alpha v+2+\alpha+7uv+u\\
& &+4\alpha u)(\alpha v-2v +\alpha-2+7uv+u-4\alpha u)(v-1)).
\end{eqnarray*}}


{\small
\begin{thebibliography}{99}
\bibitem{Conv} J. M. Couveignes, Calcul  et rationalit\'e de fonctions de Belyi, \emph{Annales de l'institut Fourier,} tome $44$, \no $1$ ($1994$), p. $1-38$.
%\bibitem{atiya} M. F. Atiyah, \emph{Complex fibre bundles and ruled surfaces}, Trans. Amer. Math. Soc. \textbf{85}
%(1957), 181Ð207.
\bibitem{B} G. V. Belyi, Galois extensions of a maximal cyclotomic field. (Russian), \emph{Izv. Akad. Nauk SSSR Ser. Mat. \textbf{43} (1979), no. 2, 267-276,} English translation in \emph{Math. USSR Izv. \textbf{14} (1980),} 247-256.
\bibitem{P1}  P. Boalch, From Klein to Painlev\'e via Fourier, Laplace and Jimbo, \emph{Proc. London Math. Soc.} ($3$) \textbf{90}($2005$),  p. $167 - 208$.
  
\bibitem{P2} P. Boalch, The fifty-two icosahedral solutions to Painlev\'e $VI$, \emph{J. Reine Angew. Math.} \textbf{596} ($2006$), p. $183 - 214$.
\bibitem{P}P. Boalch, Some explicit solutions to the Riemann-Hilbert problem, \emph{IRMA Lectures in Mathematics and Theoretical Physics,} vol. \textbf{9} ($2006$) p. $85-112$.
%\bibitem{Marco} M. Brunella, \emph{Birational Geometry of foliations,} Monograph IMPA. 








\bibitem{Doran} C. F. Doran, Algebraic and Geometric Isomonodromic Deformations, \emph{J. Differential Geometry \textbf{59}} (2001), 33-85.
 \bibitem{Dubr} B. Dubrovin and M. Mazzocco, Monodromy of certain Painlev\'e VI Transcendents and Reflection Groups, \emph{Ivent. Math.} \textbf{141} ($2000$), p. $55-147$.
\bibitem{Kapov} B. Dubrovin and M. Mazzocco, Canonical structure and symmetries of the Schlesinger equations, \emph{Comm. Math. Phys.} \textbf{271}($2007$),\no. 2, p. $289-373$.



%\bibitem{oto} Forster and Otto \emph{Lectures on Riemann Surfaces}, New York : Springer, 1981. -viii- 256 p. ISBN 0-387-90617-7.






%\bibitem{Garnier} R. Garnier, \emph{Sur des \'Equations Diff\'erentielles du troisi\`eme ordre dont l'int\'egrale G\'en\'erale est Uniforme et sur une Classe d'\'Equations Nouvelles d'ordre Sup\'erieur dont l'int\'egrale G\'en\'erale a ses Points Critiques Fixes,} Ann. Sci. \'Ecole Norm. Sup. (3) 29 (1912) 1Ð126.
%\bibitem{Garnier2} R. Garnier, \emph{\'Etude de l'int\'egrale g\'en\'erale de l'\'equation VI de M. Painlev\'e dans le voisinage de ses singularit\'es transcendantes,} Ann. Sci. \'Ecole Norm. Sup. 3\ieme{} s\'erie, tome \mbox{34} (\mbox{1917}), p. 239-353.
\bibitem{Garnier1} R. Garnier, Sur des syst\`emes diff\'erentiels du second ordre dont l'int\'egrale g\'en\'erale est uniforme, \emph{Ann. Sci. \'Ecole Norm. Sup.} 3\ieme{} s\'erie, tome 77, \no 2 (\mbox{1960}), p. 123-144.

%\bibitem{Jone} Gareth, A. Jones and David Singerman, \emph{Complex Functions : An Algebraic and Geometric viewpoint}, Cambridge university Press, 1987. ISBN 0-521-31366-X.
 \bibitem{Hitchin} N. J. Hitchin, Poncelet Polygons and the Painlev\'e Equations, \emph{Geometry and analysis (Bombay, 1992), Tata
Inst. Fund. Res., Bombay}, $1995$, p.$151-185$.
  \bibitem{Hitchin1} N. J. Hitchin, Twistor spaces, Einstein metrics and isomonodromic deformations, \emph{J. Differential Geom.} 42($1995$), \no. 1, p. $30-112$.
  
% \bibitem{Gu} R. C. Gunning, \emph{Special Coordinate Coverings of Riemann Surfaces}, Math. Annalen \textbf{170}, 67Ð86 (1967).
 
 
% \bibitem{Viktoria} V. Heu, \emph{Isomonodromic deformations and maximally stable bundles,} hal-00308586, 2008.
\bibitem{Hodgkinson} Mr. J. Hodgkinson, A detail in Conformal Representation, \emph{Proc. London Math. Soc.,} Ser. $2$, Vol. $15$($1916$), p. $166-181$.




%\bibitem{Ice} E. L. Ince, \emph{Ordinary Differential Equations}, 1926, pp. 317 - 356, ISBN 0-486-60349-0.
\bibitem{Isawaki} K. Iwasaki, H. Kimura, S. Shimomura, S. Yoshida, From Gauss to Painlev\'e: A Modern Theory of Special Functions, \emph{Braunschweig : Vieweg}, ($1991$), $347$pages.








\bibitem{Jimbo} Jimbo and T. Miwa, Monodromy preserving deformation of linear ordinary differential equations with rational coefficients \textbf{$II$}, \emph{Physica \textbf{2D}} ($1981$), p.$407-448$. 









\bibitem{Ki} A. V. Kitaev and R. Vidanus, Transformations $RS^2_4(3)$ of the Ranks $\le 4$ and Algebraic Solutions of the Sixth Painlev\'e Equation, \emph{Comm. Math. Phys.} \textbf{228} ($2002$), p.$151-176$.
\bibitem{Belyi} A.~Kitaev, Grothendieck's Dessin d'Enfants, Their Deformations and Algebraic Solutions of the Sixth Painlev\'e and Gauss Hypergeometric Equations, \emph{Algebra i Analiz \textbf{17}}, \no. 1 ($2005$), p.$224-273$.
\bibitem{K2}A. V. Kitaev, Remarks Towards the Classification of $RS^2_4(3)-$Transformations and Algebraic Solutions of the Sixth Painlev\'e Equation, S\'emin. Congr., 14, \emph{Soc. Math. France}, Paris, ($2006$), p. $199-227$.

\bibitem{Klein} F.~Klein, Vorlesungen \"uber das Ikosaedar, \emph{B. G. Teubner, Leipzig}, ($1884$).


\bibitem{Lisovyy} O. Lisovyy, Y. Tykhyy, Algebraic Solutions of the sixth Painlev\'e Equation, Preprint \emph{http://arxiv.org/abs/$0809.4873$v2} ($2008$)
%\bibitem{Frank} F. Loray, \emph{Painlev\'e $VI$ \'Equation And the Riemann-Hilbert Correspondance} monograph.
%\bibitem{Lame} F. Loray,  \emph{Okamoto symmetry of Painlevé VI equation and isomonodromic deformation 
%of Lam\'e connections,} RIMS Kôkyûroku Bessatsu B2 (2007) p.129-136. 
%\bibitem{Loray} F. Loray and D. Marin, \emph{Projective Structures and Projective Bundles over compact Riemann Surfaces,} arxiv: 0706.3608v1, (2007).
%\bibitem{LorayPereira} F. Loray and J.V. Pereira, \emph{Transversely projective foliations on surfaces:
%existence of minimal form and prescription of monodromy,}
%Intern. Jour. Math. 18 (2007) p.723-747. 





%\bibitem{Machu} F. X. Machu, \emph{Monodromy of a class of Logarithmic Connections on an Elliptic Curve}, Symmetry, Integrability and Geometry: Methods and Applications, SIGMA 3(2007), 082, 31 pages.





%\bibitem{Naga} M. Nagata, \emph{On self-intersection number of a section on a ruled surface}, Nagoya Math. J. 37 (1970), 191-196.



%\bibitem{Pak} F. Pakovich, \emph{Solution of the Hurwitz problem for Laurent polynomials}, arxiv: math/0611776v2 [math.GT], 2007.
%\bibitem{Poinc} H. Poincar\'e, \emph{Sur les groupes des \'equations lin\'eaires,} Acta Math. 4 (1884), 201Ð312.

\bibitem{Uniformisation} H. P. de Saint-Gervais, Uniformisation des surfaces de Riemann. Retour sur un théorème centenaire, \emph{ENS \'Editions \'Ecole normale sup\'erieure de Lyon ($2011$)} $544$ pages.

%\bibitem{Sch} L. Schlesinger. Ueber eine Klasse von Differentsial System Beliebliger Ordnung mit Festen Kritischer Punkten. \emph{J. fur Math.}, 141:96Ð145, 1912.
\bibitem{Watanabe}H. Watanabe, Birational canonical transformations and classical solutions of the sixth Painlev\'e
equation, \emph{Ann. Scuola Norm. Sup. Pisa Cl. Sci. 27}, ($1999$), p.$379-425$.



%\bibitem{V1} R. Vidunas, \emph{Algebraic Transformations of Gauss hypergeometric Functions,} preprint http://www.arxiv.org/math.CA/0408269(2004).
%\bibitem{V2} R. Vidunas, Transformations of some Gauss hypergeometric Functions,  \emph{J. Comp. Appl. Math. \textbf{178} (2005),} 473-487.


\end{thebibliography}
}
\end{document}